\documentclass[11pt]{amsart}
\usepackage{amsfonts, amssymb, amsmath, amsthm, color, float,enumerate}

\theoremstyle{plain}
   \newtheorem{teo}{Theorem}
   
   \newtheorem{lema}[teo]{Lemma}

\theoremstyle{definition}
   
\theoremstyle{remark}

 \newtheorem{afirmacion}{Claim}

\numberwithin{equation}{section}
\allowdisplaybreaks

\newcommand{\R}{\mathbb{R}} 
\newcommand{\Z}{\mathbb{Z}} 
\newcommand{\e}{\varepsilon} 
\newcommand{\s}[2]{\sum_{#1}^{#2}} 
\newcommand{\norm}[1]{\|#1\|} 

\begin{document}

\title[Mixed weak estimates for maximal operators]{Mixed weak estimates of Sawyer type for generalized maximal operators}

\author[F. Berra]{Fabio Berra}
\address{CONICET and Departamento de Matem\'{a}tica (FIQ-UNL),  Santa Fe, Argentina.}
\email{fberra@santafe-conicet.gov.ar}



\thanks{The author was supported by CONICET and UNL}

\subjclass[2010]{42B20, 42B25}

\keywords{Young functions, maximal operators, Muckenhoupt weights}


\begin{abstract}
We study mixed weak estimates of Sawyer type for maximal operators associated to the family of Young functions $\Phi(t)=t^r(1+\log^+t)^{\delta}$, where $r\geq 1$ and $\delta\geq 0$. More precisely, if $u$ and $v^r$ are $A_1$ weights, and $w$ is defined as $w=1/\Phi(v^{-1})$ then the following estimate
\[uw\left(\left\{x\in \R^n: \frac{M_\Phi(fv)(x)}{v(x)}>t\right\}\right)\leq C\int_{\R^n}\Phi\left(\frac{|f(x)|v(x)}{t}\right)u(x)\,dx\]
holds for every positive $t$. This extends mixed estimates to a wider class of maximal operators, since when we put $r=1$ and $\delta=0$ we recover a previous result for the classical Hardy-Littlewood maximal operator.

This inequality generalizes the result proved by Sawyer in Proc. Amer. Math. Soc. \textbf{93} (1985), no.~4, 610--614. Moreover, it includes estimates for some maximal operators related with commutators of Calder\'on-Zygmund operators.
\end{abstract}

\maketitle

\thispagestyle{empty}

\section*{Introduction}

Mixed inequalities were introduced by E. Sawyer in \cite{Sawyer}. These inequalities include two Muckenhoupt weights $u$ and $v$ and the weak type estimate that Sawyer proved in that article involves a level set of certain operator, which is related with the classical Hardy-Littlewood maximal operator. This modification implies that the most known techniques involved with classical weak type inequalities must be replaced for other ones. Concretely, Sawyer proved that if $u,v$ are weights belonging to the $A_1$-Muckenhoupt class then the inequality

\begin{equation}\label{mixta_para_M}
uv\left(\left\{x\in \R: Sf(x)>t\right\}\right)\leq \frac{C}{t}\int_{\R}|f(x)|u(x)v(x)\,dx
\end{equation}
holds for every positive $t$, where $S(f)(x)=M(fv)(x)v^{-1}(x)$.

The motivation of Sawyer of considering \eqref{mixta_para_M} was the fact that this inequality together with the Jone's factorization Theorem allow us to obtain an alternative proof of the boundedness of the Hardy-Littlewood maximal operator $M$ in $L^p(w)$, provided $w\in A_p$, $1<p<\infty$.

Later on, Cruz Uribe, Martell and P\'erez proved in \cite{CU-M-P} an extension of the Sawyer estimate on the real line to higher dimensions. They showed that if $u,v$ are weights that satisfy $u,v\in A_1$ or $u\in A_1$ and $v\in A_\infty(u)$ then the inequality

\begin{equation}\label{mixta_para_M_y_T_Rn}
uv\left(\left\{x\in \R^n: \frac{|\mathcal{T}(fv)(x)|}{v(x)}>t\right\}\right)\leq \frac{C}{t}\int_{\R^n}|f(x)|u(x)v(x)\,dx
\end{equation}
holds for every positive $t$, where $\mathcal{T}$ is either the Hardy-Littlewood maximal function or a Calder\'on-Zygmund operator. We want to point out that these authors not only extended Sawyer's estimate to higher dimensions and other operators but also included another condition on the weights. This condition is ``smoother'' than $u,v\in A_1$ since it can be shown that it implies $uv\in A_\infty$ while, in the other case, that product might be very singular.

Recently, in \cite{Li-Ombrosi-Perez}, Li, Ombrosi and P\'erez extended   these estimates to a more general context. Concretely, they proved that if $u\in A_1$ and $v\in A_\infty$ then the inequality \eqref{mixta_para_M_y_T_Rn} holds for every positive $t$.

Then, a natural question that arises is if similar estimates hold for other maximal operators, which are defined by means of certain Young functions. Particularly, we consider $L$log$L$ type functions since, it is well known, they provide maximal functions related with commutators of Calder\'on-Zygmund operators.

In \cite{bcp} the authors proved a mixed weighted inequality for such operators, but for a particular weight $v(x)=|x|^{-\beta}$ with $\beta<-n$. This means that $v$ is not even locally integrable. No assumptions were made on $u$. More specifically, they proved that if $u\geq 0$, $v$ is as above and $w=1/\Phi(v^{-1})$ then

\begin{equation*}
uw\left(\left\{x\in \R^n: \frac{M_\Phi(fv)(x)}{v(x)}>t\right\}\right)\leq C\int_{\R^n}\Phi\left(\frac{|f(x)|v(x)}{t}\right)Mu(x)\,dx
\end{equation*}
holds for every positive $t$, where $\Phi(t)=t^r(1+\log^+t)^{\delta}$, with $r\geq 1$, $\delta\geq0$ and $\log^+t=\max\{0,\log t\}$, for $t>0$.

Notice that the product $uv$ is replaced on this last estimate by $uw$. This fact suggests us that if we consider maximal operators associated with Young functions the external weight $v$ should be modified in that way. This seems to be a well extension, since if we take $r=1$ and $\delta=0$, the operator $M_\Phi$ is the Hardy-Littlewood maximal operator and $w=v$, so we recover estimate \eqref{mixta_para_M_y_T_Rn}.

Let us also point out that if we consider the operator $M_r(f)=M(f^r)^{1/r}$, $r\geq 1$ we have that

\[\left\{x: \frac{M_r(fv)(x)}{v(x)}>t\right\}=\left\{x: \frac{M((fv)^r)(x)}{v^r(x)}>t^r\right\}.\]
So if $u,v^r$ are $A_1$ weights then inequality \eqref{mixta_para_M_y_T_Rn} yields
\begin{equation*}
uv^r\left(\left\{x\in \R^n: \frac{M_r(fv)(x)}{v(x)}>t\right\}\right)\leq \frac{C}{t^r}\int_{\R^n}|f(x)|^ru(x)v^r(x)\,dx.
\end{equation*}
But this last inequality can be also written as follows
\begin{equation}\label{mixed_Mr_segundaforma}
uw\left(\left\{x\in \R^n: \frac{M_\Phi(fv)(x)}{v(x)}>t\right\}\right)\leq C\int_{\R^n}\Phi\left(\frac{|f(x)|v(x)}{t}\right)u(x)\,dx,
\end{equation}
where $w=1/\Phi(v^{-1})$ with $\Phi(t)=t^r$.

\medskip

Thus, as in the case of the operator $M$, the inequality above allow us to obtain an alternative proof of the boundedness of the operator $M_r$ in $L^p(w)$, with $r<p<\infty$ and $w$ belonging to the $A_{p/r}$-Muckenhoupt class. Indeed, given $w\in A_{p/r}$ we use the Jone's factorization Theorem to decompose it as $w=uv^{r-p}=uv^{r(1-p/r)}$, where $u,v^r\in A_1$, and consider the auxiliar operator $S(f)(x)=M_r(fv)(x)v^{-1}(x)$. Then, \eqref{mixed_Mr_segundaforma} is the $(r,r)$-weak type inequality of the operator $S$ with respect to the measure $d\mu=uv^r\,dx$. Also, it is not difficult to see that $S$ is bounded in $L^\infty(uv^r)$. Then, by using the Marcinkiewicz's interpolation theorem we get the boundedness of $S$ in $L^p(uv^r)$, $r<p<\infty$ and thus

\begin{align*}
\int_{\R^n}M_r(f)(x)^pw(x)\,dx&=\int_{\R^n}M_r(f)^p(x)u(x)v^{r-p}(x)\,dx\\
&=\int_{\R^n}S(fv^{-1})^p(x)u(x)v^r(x)\,dx\\
&\leq C\int_{\R^n}|f(x)|^pu(x)v^{r-p}(x)\,dx\\
&=C\int_{\R^n}|f(x)|^pw(x)\,dx.
\end{align*}

\medskip

In this paper we shall consider a wider class of maximal operators that includes the operators $M_r$, $r\geq 1$ and we prove that they satisfy an analogous inequality, under the same condition on the weights, i.e., $u,v^r\in A_1$. Concretely, we will prove the following theorem.
\begin{teo}\label{teo_main}
Let $r\geq 1$, $\delta\geq0$ and $\Phi(t)=t^r(1+\log^+t)^\delta$. If $u,v^r$ are $A_1$ weights and $w=1/\Phi(v^{-1})$, then there exists a positive constant $C$ such that the inequality
\begin{equation}\label{eq_teo_main}
uw\left(\left\{x\in \R^n: \frac{M_\Phi(fv)(x)}{v(x)}>t\right\}\right)\leq C\int_{\R^n}\Phi\left(\frac{|f(x)|v(x)}{t}\right)u(x)\,dx
\end{equation}
holds for every positive $t$ and every bounded function $f$ with compact support.
\end{teo}

	By virtue of the extension of \eqref{mixta_para_M_y_T_Rn} in \cite{Li-Ombrosi-Perez} to the case $u\in A_1$ and $v\in A_\infty$, we conjecture that Theorem~\ref{teo_main} should still hold for the case $u\in A_1$ and $v^r\in A_\infty$. If the conjecture is true, we could get the result in \cite{Li-Ombrosi-Perez} for the Hardy-Littlewood maximal function by taking $\Phi(t)=t$.
	
The remainder of the paper is organized as follows: in $\S$1 we give the preliminaries and definitions. In $\S$2 we prove certain lemmas we shall use in the proof of the main result, which is contained in $\S$3.

\section{Preliminaries and basic results}\label{seccion_preliminares}

Recall that a \emph{weight} $w$ is a function that is locally integrable, positive and finite in almost every $x$. Given $1<p<\infty$ we say that $w\in A_p$ if there exists a positive constant $C$ such that
\[\left(\frac{1}{|Q|}\int_Qw\,dx\right)\left(\frac{1}{|Q|}\int_Qw^{1-p'}\,dx\right)^{p-1}\leq C,\]
for every cube $Q\subseteq \R^n$. By a cube $Q$ we understand a cube in $\R^n$ with sides parallel to the coordinate axes.
For $p=1$ we say that $w\in A_1$ if there exists a positive constant $C$ such that for every cube $Q$
\[\frac{1}{|Q|}\int_Qw\,dx\leq C\inf_Qw.\]

The smallest constant $C$ for which the inequalities above hold is denoted by $[w]_{A_p}$ and called the $A_p$ constant of $w$.

Finally, the $A_\infty$ class is defined as the collection of all the $A_p$ classes, that is, $A_\infty=\bigcup_{p\geq 1}A_p$. It is well known that $A_p$ classes are increasing on $p$, that is, if $p\leq q$ then $A_p\subseteq A_q$.  For more details and other properties of weights see \cite{javi} or \cite{grafakos}.

There are many conditions that characterize the set $A_\infty$. In this paper we will use the following one: we say that $w\in A_\infty$ if there exist positive constants $C$ and $\e$ such that, for every cube $Q\subseteq \R^n$ and every measurable set $E\subseteq Q$ the condition
\[\frac{w(E)}{w(Q)}\leq C\left(\frac{|E|}{|Q|}\right)^{\e}\]
holds, where $w(E)=\int_E w$. With this characterization we obtain the following result, which was previously proved in \cite{M-W-Maximal-Hilbert} for $n=1$. For the sake of completeness we include the proof.

\begin{lema}\label{estimacion_lema_M-W}
Let $w\in A_\infty$, $\lambda>0$ and $Q$ a cube of $\R^n$. Then, there exist positive constants $C_0$ and $\xi$ such that
\[|\{w(x)>\lambda\}\cap Q|\leq C_0|Q|\left[\frac{1}{\lambda|Q|}\int_Q w(x)\,dx\right]^{1+\xi}.\]
\end{lema}


\begin{proof}
Since $w\in A_\infty$, there exist positive constants $C$ and $\e$ such that for every cube $Q$ and every measurable set $E\subseteq Q$
\begin{equation}\label{condicion_A_infinito}
\frac{w(E)}{w(Q)}\leq C\left(\frac{|E|}{|Q|}\right)^{\e}.
\end{equation}
Let $E_\lambda=\{w(x)>\lambda\}\cap Q$. Thus $|E_\lambda|\leq\frac{1}{\lambda}\int_{E_{\lambda}}w(x)\,dx$. By \eqref{condicion_A_infinito} we obtain
\[|E_{\lambda}|\leq \frac{1}{\lambda}w(E_\lambda)\leq \frac{C}{\lambda}\left(\frac{|E_\lambda|}{|Q|}\right)^\e w(Q),\]
which implies that
\[|E_\lambda|\leq |Q|\left(\frac{C}{\lambda|Q|}w(Q)\right)^{1/(1-\e)}.\]
Taking $\xi=1/(1-\e)-1$ and $C_0=C^{1/(1-\e)}$ the desired estimate follows.
\end{proof}

\begin{lema}\label{relacion_phi_con_potencia}
If $\Phi(t)=t^r(1+\log^+t)^{\delta}$, for $t\geq 1$ and $\e>0$ we have that
\[\Phi(t)\leq Ct^{r+\e},\]
with $C=\max\{(\delta/\e)^{\delta},1\}$.
\end{lema}

\begin{proof}
The proof of the inequality $t^r(1+\log^+t)^{\delta}\leq Ct^{r+\e}$ for $t\geq 1$ can be achieved by showing that $1\leq C^{1/\delta}t^{\e/\delta}-\log t$ holds for every $t\geq 1$. Let $f(t)=C^{1/\delta}t^{\e/\delta}-\log t$ and note that $f(1)=C^{1/\delta}\geq 1$. Then, for $t>1$,
\[f'(t)=\frac{\e}{\delta}C^{1/\delta}\frac{t^{\e/\delta}}{t}-\frac{1}{t}\]
and $f'(t)>0$ if and only if $t>(\delta/\e)^{\delta/\e}C^{-1/\e}$ and this inequality is always true because of the definition of $C$. So $f$ is an increasing function and the inequality above holds.
\end{proof}

\bigskip

Given a locally integrable function $f$, the \emph{Hardy-Littlewood maximal operator} is defined by
\[Mf(x)=\sup_{Q\ni x}\frac{1}{|Q|}\int_Q |f(y)|\,dy.\]

We say that $\varphi:[0,\infty)\to[0,\infty]$ is a \emph{Young function} if it is convex, increasing, $\varphi(0)=0$ and $\varphi(t)\to\infty$ when $t\to\infty$. Given a Young function $\varphi$, the maximal operator $M_\varphi$ is defined, for $f\in L^1_{\textit{loc}}$, by
\[M_\varphi f(x)=\sup_{Q\ni x}\norm{f}_{\varphi,Q},\]
where $\norm{f}_{\varphi,Q}$ denotes the \emph{average of Luxemburg type} of the function $f$ in the cube $Q$, which is defined as follows
\[\norm{f}_{\varphi,Q}=\inf\left\{\lambda>0 : \frac{1}{|Q|}\int_Q\varphi\left(\frac{|f(y)|}{\lambda}\right)\,dy\leq 1 \right\}.\]

\bigskip

By a \emph{dyadic grid} $\mathcal{D}$ we will understand a collection of cubes of $\R^n$ that satisfies the following properties
\begin{enumerate}
\item all cube $Q$ in $\mathcal{D}$ has side length $2^k$, for some $k\in\Z$,
\item if $P\cap Q\neq\emptyset$ then $P\subseteq Q$ or $Q\subseteq P$,
\item $\mathcal{D}_k=\{Q\in\mathcal{D}: \ell(Q)=2^k\}$ is a partition of $\R^n$ for every $k\in\Z$, where $\ell(Q)$ denotes the length of each side of $Q$.
\end{enumerate}

To a given dyadic grid $\mathcal{D}$ we can associate the corresponding maximal operator $M_{\varphi,\mathcal{D}}$ defined similarly as above, where the supremum is taken over all cube in $\mathcal{D}$. When $\varphi(t)=t$, we will simply denote  this operator with $M_{\mathcal{D}}$.

\medskip

The next result will be useful in our estimates. A proof can be found in \cite{Okikiolu}.

\begin{teo}\label{teo_control_diadico}
There exist dyadic grids $\mathcal{D}^{(i)}$, $1\leq i\leq 3^n$ such that for every cube $Q\subseteq\R^n$ there exist $i$ and $Q_0\in \mathcal{D}^{(i)}$ such that $Q\subseteq Q_0$ and $\ell(Q_0)\leq 3\ell(Q)$.
\end{teo}

With this result in mind, it will be sufficient to prove Theorem~\ref{teo_main} for $M_{\Phi,\mathcal{D}}$, for a general dyadic grid $\mathcal{D}$, since
\[M_{\Phi}f(x)\leq C\s{i=1}{3^n}M_{\Phi,\mathcal{D}^{(i)}}f(x).\]

Indeed, fix $x\in\R^n$ and $Q$ a cube containing $x$. By Theorem~\ref{teo_control_diadico} we have a dyadic grid $\mathcal{D}^{(i)}$ and $Q_0\in \mathcal{D}^{(i)}$ with the properties above. Then,
\[\frac{1}{|Q|}\int_Q \Phi\left(\frac{|f(y)|}{\norm{f}_{\Phi,Q_0}}\right)\,dy \leq\frac{|Q_0|}{|Q|}\frac{1}{|Q_0|}\int_{Q_0}\Phi\left(\frac{|f(y)|}{\norm{f}_{\Phi,Q_0}}\right)\,dy
\leq 3^n,\]
so
\[\norm{f}_{\Phi,Q}\leq 3^n\norm{f}_{\Phi,Q_0}\leq 3^n M_{\Phi, \mathcal{D}^{(i)}}(f)(x)\leq 3^n\s{i=1}{3^n}M_{\Phi, \mathcal{D}^{(i)}}(f)(x).\]
Thus, by taking supremum over all cubes $Q$ that contain $x$ we have the desired estimate.

\section{Previous lemmas}

In this section we will state and prove some lemmas that will be useful in the proof of our main result.

\begin{lema}\label{descomposicion_de_CZ_del_espacio}
Given $\lambda>0$, a bounded function with compact support  $f$, a dyadic grid $\mathcal{D}$ and a Young function $\varphi$, there exists a family of dyadic cubes $\{Q_j\}_j$ of $\mathcal{D}$ that satisfies
\[\{x\in\R^n: M_{\varphi,\mathcal{D}}f(x)>\lambda\}=\bigcup_j Q_j,\]
and $\norm{f}_{\varphi,Q_j}>\lambda$  for every $j$.
\end{lema}

\begin{proof}
Let $\lambda>0$ be fixed. For $k\in \Z$ we define $E_kf(x)=\s{Q\in \mathcal{D}_k}{}\norm{f}_{\varphi,Q}\mathcal{X}_Q(x)$. With this definition, we can write
$M_{\varphi,\mathcal{D}}f(x)=\sup_k E_kf(x)$. Next, we consider the sets
\[\Lambda_k=\{x\in\R^n: E_kf(x)>\lambda \textrm{ and } E_jf(x)\leq \lambda \textrm{ if } j>k \}.\]
Then, we have that $\{x\in\R^n: M_{\varphi,\mathcal{D}}f(x)>\lambda\}=\bigcup_k \Lambda_k$. Indeed, if $x\in \{M_{\varphi,\mathcal{D}}f(x)>\lambda\}$ there exist $k\in \Z$ and $Q\in\mathcal{D}_k$ such that $\norm{f}_{\varphi,Q}>\lambda$ and this means that $E_kf(x)>\lambda$. Notice that $E_kf(x)\to 0$ when $k\to \infty$ since the Luxemburg norm $\norm{f}_{\Phi,Q}$ tends to zero when $|Q|\to\infty$, because $f$ is bounded and has compact support.


Then, there exists $k_0\in\Z$ such that if $j>k_0$, $E_jf(x)\leq \lambda$. Now, if for every $i: k<i\leq k_0$ we have $E_if(x)\leq \lambda$ then $x\in \Lambda_k$. If not, let $i_0$ the biggest integer less or equal than $k_0$ for which $E_{i_0}f(x)>\lambda$. In this case, $x\in \Lambda_{i_0}$. Conversely, if $x\in \bigcup_k \Lambda_k$ there exists $k_0$ such that $x\in\Lambda_{k_0}$ and this means that $E_{k_0}f(x)>\lambda$ which yields $M_{\varphi,\mathcal{D}}f(x)>\lambda$.

Finally, observe that every set $\Lambda_k$ can be written as a union of cubes of $\mathcal{D}_k$ with the desired property since, for a fixed $x\in\Lambda_k$, we have $y\in \Lambda_k$ for all $y\in Q(k)$, where $Q(k)$ is the cube in $\mathcal{D}_k$ that contains $x$.
\end{proof}

Notice that the way we build the sets $\Lambda_k$ ensures us that the cubes $Q_j$ are maximal in the sense of inclusion, that is, if $Q_j\subsetneq Q'$ for a fixed $j$, then $\norm{f}_{\Phi,Q'}\leq \lambda$.

\medskip

Throughout this paper, we will denote $\Phi(t)=t^r(1+\log^+t)^{\delta}$, with $r\geq 1$ and $\delta\geq0$.

\begin{lema}\label{lema_propiedad_bk}
Given a number $a>1$, for $k\in\Z$ define $b_k=1/\Phi(a^{-k})$. Then,
\[a^r\leq \frac{b_{k+1}}{b_k}\leq \Phi(a),\]
for every $k$.
\end{lema}
\begin{proof}
\begin{align*}
\frac{b_{k+1}}{b_k}&=\frac{a^{-rk}(1+\log^+a^{-k})^{\delta}}{a^{-r(k+1)}(1+\log^+a^{-(k+1)})^{\delta}}\\
&=a^r\left(\frac{1+\log^+a^{-k}}{1+\log^+a^{-(k+1)}}\right)^\delta=:a^r(\varphi_k(a))^{\delta}.
\end{align*}
Let us notice that, if $k=-1$, $\varphi_k(a)=(1+\log^+a)$. If $k\geq 0$, then $\varphi_k(a)=1$, and if $k<-1$ we have
\begin{align*}
1+\log^+(a^{-k})&=1+\log^+a^{-(k+1)}a\\
&=1+\log{a^{-(k+1)}}+\log a,
\end{align*}
thus $\varphi_k(a)=1+\frac{\log a}{1+\log^+ a^{-(k+1)}}$. So we can deduce that
\[1\leq \varphi_k(a)\leq 1+\log^+ a,\]
and by raising every member to the power $\delta$ and by multiplying by $a^r$ we are done.
\end{proof}

We shall devote the end of this section to prove some results concerning to the following set, which is essential in the proof of our main result.
For a fixed $a>1$ and for each
 $k\in\Z$, set
\[\Omega_k=\{x\in\R^n: M_{\mathcal{D}}v(x)>a^k\}\cap\{x\in\R^n: M_{\Phi,\mathcal{D}}g(x)>a^k\},\]
where $\Phi(t)=t^r(1+\log^+t)^{\delta}$, $v^r\in A_1$ and $g$ is a function that we define later in the corresponding proof. For each $k$,  $\Omega_k$ can be written as the disjoint union of dyadic maximal cubes $\{Q^k_j\}_j$ contained in a dyadic grid $\mathcal{D}$. Indeed, from Lemma~\ref{descomposicion_de_CZ_del_espacio} each set can be written in that way separately. Thus,
\begin{equation}\label{descomposicion_Omega_k}
\Omega_k=\left(\bigcup_\ell R_\ell^k\right)\cap\left(\bigcup_i S_i^k\right)=\bigcup_{\ell,i}R_\ell^k\cap S_i^k=\bigcup_j Q_j^k,
\end{equation}
where $Q_j^k=S_i^k$ if $S_i^k\subset R_\ell^k$ and $Q_j^k=R_\ell^k$ otherwise,  for each $(\ell,i)$ for which the intersection is nonempty. For these cubes we have that

\begin{equation}\label{propiedad_cubos_para_v_sin_gamma}
\frac{a^k}{[v]_{A_1}}\leq \frac{1}{[v]_{A_1}}\inf_{Q^k_j}M_\mathcal{D}v\leq \inf_{Q^k_j}v.
\end{equation}

\medskip

The next two lemmas deal with the dyadic maximal cubes $\{Q^k_j\}_j$ that decompose  $\Omega_k$.

\begin{lema}\label{lema_acotacion_promedios_vk}
Let $k\in\Z$, $v_k(x)=\min\{v^r(x),b_{k+1}\}$ with $v\in A_1$. If $Q^\ell_j$ is a cube as in \eqref{descomposicion_Omega_k} with $\ell\geq k$ then
\[\frac{b_k}{[v]_{A_1}^r}\leq \frac{1}{|Q^\ell_j|}\int_{Q^\ell_j}v_k(x)\,dx\leq b_{k+1},\]
where $b_k=1/\Phi(a^{-k})$.
\end{lema}

\begin{proof}
From the definition of $v_k$, we trivially have the second inequality. To see that the first one holds, let us consider the subsets of $Q^\ell_j$ given by $A=\{x\in Q^\ell_j: v_k(x)=v^r(x)\}$ and $B=Q^\ell_j\backslash A$.  Notice that
\begin{align*}
\frac{1}{|Q^\ell_j|}\int_{Q^\ell_j}v_k&=\frac{1}{|Q^\ell_j|}\left[\int_A v^r+\int_B b_{k+1}\right]\\
&\geq \frac{1}{|Q^\ell_j|}\left[(\inf_{A}v^r)|A|+b_{k+1}|B|\right]\\
&\geq \frac{1}{|Q^\ell_j|}\left[(\inf_{Q^\ell_j}v^r)|A|+b_k\frac{b_{k+1}}{b_k}|B|\right].
\end{align*}

From \eqref{propiedad_cubos_para_v_sin_gamma} and Lemma~\ref{lema_propiedad_bk} we have that
\begin{align*}
\frac{1}{|Q^\ell_j|}\int_{Q^\ell_j}v_k&\geq \frac{1}{|Q^\ell_j|}\left[\left(\frac{a^\ell}{[v]_{A_1}}\right)^r|A|+\frac{b_k}{[v]_{A_1}^r}a^r[v]_{A_1}^r|B|\right]\\
&\geq \frac{1}{|Q^\ell_j|}\left[\left(\frac{a^k}{[v]_{A_1}}\right)^r|A|+\frac{b_k}{[v]_{A_1}^r}a^r[v]_{A_1}^r|B|\right]\\
&\geq \frac{b_k}{[v]_{A_1}^r}\left[\frac{|A|+|B|}{|Q^\ell_j|}\right]=\frac{b_k}{[v]_{A_1}^r},
\end{align*}
where we have used that $\ell\geq k$ and $a^{rk}\geq b_k$, because of the definition of $\Phi$.
\end{proof}

%

Let us define $\Gamma=\{(k,j): |Q^k_j\cap\{x: v(x)\leq a^{k+1}\}|>0\}$. Thus, if $(k,j)\in\Gamma$ we obtain that

\begin{equation}\label{propiedad_cubos_para_v}
\frac{a^k}{[v]_{A_1}}\leq \frac{1}{[v]_{A_1}}\inf_{Q^k_j}M_\mathcal{D}v\leq \inf_{Q^k_j}v\leq \frac{1}{|Q^k_j|}\int_{Q^k_j}v\leq [v]_{A_1}\inf_{Q^k_j}v\leq[v]_{A_1}a^{k+1}.
\end{equation}

\begin{lema}\label{lema_estimacion_para_vt}
Let $Q_s^t$ be a cube such that $(t,s)\in\Gamma$, $v^r\in A_1$ and $E=Q_s^t\cap\{x: M_\mathcal{D}v(x)>a^k\}$, $k\in\Z$. Then, there exist positive constants $C>0$ and $\eta>1$ such that
\[v_t(E)\leq Cv_t(Q_s^t)a^{(t-k)r\eta}.\]
\end{lema}

\begin{proof}
Notice that $E=Q_s^t\cap\{x: (M_\mathcal{D}v(x))^r>a^{kr}\}\subseteq Q_s^t\cap\{x: v^r(x)>a^{kr}/[v]_{A_1}^r\}=:F$. Using Lemma~\ref{estimacion_lema_M-W} for $v^r\in A_1\subseteq A_\infty$, there exist $C,\e$ such that

\begin{equation}\label{eq1_lema_estimacion_para_vt}
|E|\leq|F|=\left|Q_s^t\cap\left\{x: v^r(x)>\left(\frac{a^k}{[v]_{A_1}}\right)^r\right\}\right|\leq C|Q_s^t|\left[\frac{1}{a^{kr}|Q_s^t|}\int_{Q_s^t}v^r(x)\,dx\right]^{1/(1-\e)}.
\end{equation}

Thus, given $\e$, we choose $p>1/\e$ and apply H\"{o}lder's inequality with exponents $p$ and $p'$. From the definition of $v_t$ and \eqref{eq1_lema_estimacion_para_vt} we get
\begin{align*}
\int_E v_t\,dx&\leq \left(\int_E v_t^p\,dx\right)^{1/p}|E|^{1/p'}\\
&\leq C b_{t+1}|Q_s^t|^{1/p}|Q_s^t|^{1/p'}\left[\frac{1}{|Q_s^t|a^{kr}}\int_{Q_s^t}v^r\,dx\right]^{1/(p'(1-\e))}.
\end{align*}

By using \eqref{propiedad_cubos_para_v}, Lemmas~\ref{lema_propiedad_bk} and \ref{lema_acotacion_promedios_vk} and taking $\eta=1/{p'(1-\e)}$ we have that
\begin{align*}
\int_E v_t\,dx&\leq  C \frac{b_{t+1}}{b_t}b_t|Q_s^t|a^{(t-k)r\eta}\\
&\leq C\Phi(a)v_t(Q_s^t)a^{(t-k)r\eta}\\
&=Cv_t(Q_s^t)a^{(t-k)r\eta}.
\end{align*}
\end{proof}

\section{Proof of the main result}

We devote this section to the proof of Theorem~\ref{teo_main}. It follows similar lines as in \cite{Sawyer} but with substantial changes. Since it is quite long and have some technical calculations, for the sake of clearness we will write some claims that will be proved separately.

\begin{proof}[Proof of Theorem~\ref{teo_main}]
In order to prove inequality \eqref{eq_teo_main}, fix $t>0$, a dyadic grid $\mathcal{D}$ and let $g=fv/t$. Then, it will be enough to prove that
\[uw\left(\left\{x\in \R^n: M_{\Phi,\mathcal{D}}(g)(x)>v(x)\right\}\right)\leq C\int_{\R^n}\Phi(g)\,u\,dx.\]
We can assume, without loss of generality, that $g$ is a bounded function with compact support. Fix a number $a>\max\{2^n,L\}$, where $L$ is a quantity that will be chosen later. For every $k\in\Z$ consider the numbers $a^k$ and $b_k=1/\Phi(a^{-k})$. As we said before, the set
\[\Omega_k=\{x\in\R^n: M_{\mathcal{D}}v(x)>a^k\}\cap\{x\in\R^n: M_{\Phi,\mathcal{D}}g(x)>a^k\}\]
can be written as the disjoint union of dyadic maximal cubes $\{Q^k_j\}_j$, for each $k$ (see \eqref{descomposicion_Omega_k}). 

Let us consider the set $\Gamma=\{(k,j): |Q^k_j\cap\{x: v(x)\leq a^{k+1}\}|>0\}$. Thus, for $(k,j)\in\Gamma$ we have that
\eqref{propiedad_cubos_para_v} holds.

Notice also that if $A_k=\{x: a^k<v(x)\leq a^{k+1}\}$, then for each $k$ we have
\begin{align*}
A_k\cap\{x: M_{\Phi,\mathcal{D}}g(x)>v(x)\}&\subseteq\{x: M_\mathcal{D}v(x)>a^k\}\cap\{x: v(x)\leq a^{k+1}\}\cap\{x: M_{\Phi,\mathcal{D}}g(x)>a^k\}\\
&\subseteq\bigcup_{j: (k,j)\in \Gamma}Q^k_j,
\end{align*}
except for a set of measure zero. Also, if $x\in A_k$ then $b_k<w(x)\leq b_{k+1}$. Thus, we get
\begin{align*}
uw(\{x\in\R^n: M_{\Phi,\mathcal{D}}g>v\})&=\s{k\in\Z}{}uw(\{M_{\Phi,\mathcal{D}}g>v\}\cap A_k)\\
&\leq \s{k\in\Z}{}\frac{b_{k+1}}{b_k}b_ku(\{M_{\Phi,\mathcal{D}}g>v\}\cap A_k)\\
&\leq \Phi(a)\s{k\in\Z}{}\s{j:(k,j)\in\Gamma}{}b_ku(Q^k_j)\\
&\leq \Phi(a)[v]_{A_1}^r\s{k\in\Z}{}\s{j:(k,j)\in\Gamma}{}u(Q^k_j)\frac{v_k(Q^k_j)}{|Q^k_j|},
\end{align*}
where we have used Lemmas~\ref{lema_propiedad_bk} and~\ref{lema_acotacion_promedios_vk} and \eqref{propiedad_cubos_para_v}.

We  fix now a negative integer $N$ and define $\Gamma_N=\{(k,j)\in \Gamma: k\geq N\}$. The objective is to prove that there exists a positive constant $C$, independent of $N$, such that
\[\s{(k,j)\in\Gamma_N}{}u(Q^k_j)\frac{v_k(Q^k_j)}{|Q^k_j|}\leq C\int_{\R^n}\Phi(g)u\,dx.\]
If the estimate above can be achieved, then the result follows by letting $N\to-\infty$.
\medskip

Let $\Delta_N=\{Q^k_j: (k,j)\in \Gamma_N\}$. Given two cubes in $\Delta_N$ they are either disjoint or one is contained in the other. Also observe that if $k>t$,
$\Omega_k\subseteq \Omega_t$, so if there exist cubes $Q^k_j$, $Q_s^t$ for which $Q^k_j\cap Q_s^t\neq\emptyset$ necessarily we must have $Q^k_j\subseteq Q_s^t$.

If $\eta>1$ is the constant that appears in Lemma~\ref{lema_estimacion_para_vt}, we choose $1<\alpha<\eta$ and define a sequence of sets by induction in the following way:
\[G_0=\{(k,j)\in \Gamma_N: Q^k_j \textrm{ is maximal in }\Delta_N\},\]
and, in a colloquial way, a pair  $(k,j)$ in $\Gamma_N$ belongs to $G_{n+1}$ if the cube $Q^k_j$ has an ``ancestor''  $Q_s^t$, with $(t,s)\in G_n$, and
$Q^k_j$ is the ``first descendant'' in $\Gamma_N$ satisfying $\mu(Q_j^k)>\mu(Q_s^t)$, where  $\mu(Q_s^t)$ is the weighted average given by
\[\mu(Q_s^t):= \frac{b_t}{a^{\alpha r t}}\frac{1}{|Q_s^t|}\int_{Q_s^t}u(x)\,dx,\]
in the sense that $\mu(Q_i^\ell)\leq\mu(Q_s^t)$ for every $(\ell,i)\in\Gamma_N$ and $Q^k_j\subsetneq Q_i^\ell\subseteq Q_s^t$. That is, we define for $n\geq 0$, $G_{n+1}$ to be the set of pairs $(k,j)\in \Gamma_N$ such that there exists $(t,s)\in G_n$ with $Q_j^k\subsetneq Q_s^t$ and for which the inequalities

\begin{equation}\label{desigualdad1_conjuntoGn}
\frac{1}{|Q^k_j|}\int_{Q^k_j}u(x)\,dx>a^{(k-t)\alpha r} \frac{b_t}{b_k}\frac{1}{|Q_s^t|}\int_{Q_s^t}u(x)\,dx
\end{equation}
and
\begin{equation}\label{desigualdad2_conjuntoGn}
\frac{1}{|Q_i^\ell|}\int_{Q_i^\ell}u(x)\,dx\leq a^{(\ell-t)\alpha r} \frac{b_t}{b_\ell}\frac{1}{|Q_s^t|}\int_{Q_s^t}u(x)\,dx
\end{equation}
hold with $(\ell,i)\in\Gamma_N$ and $Q^k_j\subsetneq Q_i^\ell\subseteq Q_s^t$.


Notice that if $G_{n_0}=\emptyset$ for some $n_0$, then $G_{n}=\emptyset$ for every $n\geq n_0$.

Let $P=\bigcup_{n\geq 0}G_n$. If $(t,s)\in P$ we will say that the cube $Q_s^t$ is a \textit{principal cube}.

\begin{afirmacion}\label{afirmacion1}
There exists a positive constant $C$ such that
\[\s{(k,j)\in\Gamma_N}{}\frac{1}{|Q^k_j|}v_k(Q^k_j)u(Q^k_j)\leq C \s{(k,j)\in P}{}\frac{1}{|Q^k_j|}v_k(Q^k_j)u(Q^k_j).\]
\end{afirmacion}

For each fixed $k\in\Z$, let us consider the disjoint collection $\{\tilde Q_i^k\}_i$ of maximal dyadic cubes given by Lemma~\ref{descomposicion_de_CZ_del_espacio}, whose union is the set $\{x\in \R^n: M_{\Phi,\mathcal{D}}g(x)>a^k\}$.  Thus, for every $i$ it follows that
\begin{equation}
a^k<\norm{g}_{\Phi,\tilde Q_i^k},
\end{equation}
which is equivalent to
\begin{equation}\label{propiedad_promedios_g}
1<\frac{1}{|\tilde Q_i^k|}\int_{\tilde Q_i^k}\Phi\left(\frac{g(y)}{a^k}\right)\,dy.
\end{equation}
Since $Q^k_j\subseteq\{x: M_{\Phi,\mathcal{D}}g(x)>a^k\}$, for each $j$ there is a unique $i=i(j,k)$ such that $Q^k_j\subseteq \tilde Q_i^k$. By Claim~\ref{afirmacion1}, Lemma~\ref{lema_acotacion_promedios_vk} and \eqref{propiedad_promedios_g} we have that
\begin{align*}
\s{(k,j)\in \Gamma_N}{}\frac{1}{|Q^k_j|}v_k(Q^k_j)u(Q^k_j)&\leq C\s{(k,j)\in P}{}\frac{1}{|Q^k_j|}v_k(Q^k_j)u(Q^k_j)\\
&\leq C \s{(k,j)\in P}{}\frac{b_{k+1}}{b_k}b_k\frac{u(Q^k_j)}{|\tilde Q_i^k|}\int_{\tilde Q_i^k}\Phi\left(\frac{g}{a^k}\right)\,dx.
\end{align*}
Since $\Phi$ is submultiplicative and from the definition of $b_k$ we obtain
\begin{align*}
\s{(k,j)\in \Gamma_N}{}\frac{1}{|Q^k_j|}v_k(Q^k_j)u(Q^k_j)&\leq C\s{(k,j)\in P}{}\Phi(a)\frac{1}{\Phi(a^{-k})}\Phi(a^{-k})\frac{u(Q^k_j)}{|\tilde Q_i^k|}\int_{\tilde Q_i^k}\Phi(g)\,dx\\
&=C\int_{\R^n}\left[\s{(k,j)\in P}{}\frac{1}{|\tilde Q_i^k|}u(Q^k_j)\mathcal{X}_{\tilde Q_i^k}(x)\right]\Phi(g(x))\,dx\\
&=C\int_{\R^n}h(x)\Phi(g(x))\,dx,
\end{align*}
where $h(x)=\s{(k,j)\in P}{}|\tilde Q_i^k|^{-1}u(Q^k_j)\mathcal{X}_{\tilde Q_i^k}(x)$.

In order to finish, it only remains to show that there exists $C>0$ such that $h(x)\leq Cu(x)$. For a given $x\in \R^n$, we can assume that $u(x)<\infty$. For every fixed $k$ there exists at most one $\tilde Q_i^k$ that verifies $x\in \tilde Q_i^k$. If so, we denote it $\tilde Q^k$ and for every $k$ we define $P_k=\{(k,j)\in P: Q^k_j\subseteq \tilde Q^k\}$ and $G=\{k: P_k\neq\emptyset\}$. Recall that $k\geq N$, and therefore $G$ is bounded from below. Let $k_0$ be the smallest element in $G$. We will build a sequence in $G$ in the following way: chosen $k_m$, for $m\geq 0$ we select $k_{m+1}$ the smallest integer in $G$ greater than $k_m$ satisfying
\begin{equation}\label{desigualdad1_sucesionkm}
\frac{1}{|\tilde Q^{k_{m+1}}|}\int_{\tilde Q^{k_{m+1}}}u(y)\,dy>\frac{2}{|\tilde Q^{k_m}|}\int_{\tilde Q^{k_m}}u(y)\,dy.
\end{equation}
It is clear that, if $k_m\leq \ell< k_{m+1}$, then
\begin{equation}\label{desigualdad2_sucesionkm}
\frac{1}{|\tilde Q^{\ell}|}\int_{\tilde Q^{\ell}}u(y)\,dy\leq\frac{2}{|\tilde Q^{k_m}|}\int_{\tilde Q^{k_m}}u(y)\,dy.
\end{equation}

The sequence $\{k_m\}_{m\geq0}$ defined above has only a finite number of terms. Indeed, if it was not the case, by applying condition \eqref{desigualdad1_sucesionkm} repeatedly,  we would have
\[[u]_{A_1}u(x)\geq \frac{1}{|\tilde Q^{k_m}|}\int_{\tilde Q^{k_m}}u(y)\,dy>2^m\frac{1}{|\tilde Q^{k_0}|}\int_{\tilde Q^{k_0}}u(y)\,dy\]
for every $m>0$, and taking limit when $m\to\infty$ we would get a contradiction. Therefore $\{k_m\}=\{k_m\}_{m=0}^{m_0}$.

With this fact in mind and denoting $F_m=\{\ell\in G: k_m\leq \ell<k_{m+1}\}$ we can write

\begin{align*}
h(x)&=\s{(k,j)\in P}{}\frac{1}{|\tilde Q_i^k|}u(Q^k_j)\mathcal{X}_{\tilde Q_i^k}(x)\\
&=\s{(k,j)\in P}{}\frac{u(Q^k_j)}{u(\tilde Q^k)}\left(\frac{1}{|\tilde Q^k|}\int_{\tilde Q^k}u(y)\,dy\right)\\
&=\s{m=0}{m_0}\s{\ell\in F_m}{}\left(\frac{1}{|\tilde Q^\ell|}\int_{\tilde Q^\ell}u(y)\,dy\right)\s{j:(\ell,j)\in P_\ell}{}\frac{u(Q_j^\ell)}{u(\tilde Q^\ell)}\\
&\leq 2\s{m=0}{m_0}\left(\frac{1}{|\tilde Q^{k_m}|}\int_{\tilde Q^{k_m}}u(y)\,dy\right)\s{\ell\in F_m}{}\s{j: (\ell,j)\in P_\ell}{}\frac{u(Q_j^\ell)}{u(\tilde Q^\ell)},
\end{align*}
where in the last inequality we have used condition \eqref{desigualdad2_sucesionkm}.

\begin{afirmacion}\label{afirmacion2}
There exists a positive constant $C$ such that
\[\s{\ell\in F_m}{}\s{j:(\ell,j)\in P_\ell}{}\frac{u(Q_j^\ell)}{u(\tilde Q^\ell)}\leq C.\]
\end{afirmacion}

If this claim holds, we are done. Indeed, denoting $C_m=|\tilde Q^{k_m}|^{-1}\int_{\tilde Q^{k_m}}u$  and using the estimate above we have that
\begin{align*}
h(x)&\leq C\s{m=0}{m_0}C_m\leq C\s{m=0}{m_0} C_{m_0}2^{m-m_0}\\
&= CC_{m_0}2^{-m_0}\s{m=0}{m_0}2^m=CC_{m_0}2^{-m_0}(2^{m_0+1}-1)\\
&\leq CC_{m_0}\leq C[u]_{A_1}u(x).
\end{align*}
\end{proof}

In order to conclude, we prove Claim~\ref{afirmacion1} and \ref{afirmacion2}.

\bigskip

\begin{proof}[Proof of Claim~\ref{afirmacion1}]
Fix $(t,s)\in P$ and define
 \[
I(t,s)=\{(k,j)\in \Gamma_N: Q^k_j\subseteq Q_s^t \textrm{ and $Q_s^t$ is the smallest principal cube that contains } Q^k_j\}.\]
 Particularly, every $Q^k_j$ with $(k,j)\in I(t,s)$ is not principal, unless $(k,j)=(t,s)$. From condition \eqref{desigualdad2_conjuntoGn} we can write
 \[\s{(k,j)\in I(t,s)}{}\frac{v_k(Q^k_j)}{|Q^k_j|}u(Q^k_j)\leq \s{(k,j)\in I(t,s)}{}a^{(k-t)\alpha r}\frac{b_t}{b_k}\frac{u(Q_s^t)}{|Q_s^t|}\frac{v_k(Q^k_j)}{v_t(Q^k_j)}v_t(Q^k_j).\]
From Lemma~\ref{lema_acotacion_promedios_vk} with $k>t$ we have that
\[\frac{v_k(Q^k_j)}{v_t(Q^k_j)}\leq [v]_{A_1}^r\frac{b_{k+1}}{b_t}\leq [v]_{A_1}^r\Phi(a)\frac{b_k}{b_t},\]
so we can write
\begin{align*}
\s{(k,j)\in I(t,s)}{}\frac{v_k(Q^k_j)}{|Q^k_j|}u(Q^k_j)& \leq\Phi(a)[v]_{A_1}^r\frac{u(Q_s^t)}{|Q_s^t|}\s{(k,j)\in I(t,s)}{}a^{(k-t)\alpha r}\frac{b_t}{b_k}\frac{b_k}{b_t}v_t(Q^k_j)\\
&\leq \Phi(a)[v]_{A_1}^r \frac{u(Q_s^t)}{|Q_s^t|}\s{(k,j)\in I(t,s)}{}a^{(k-t)\alpha r}v_t(Q_s^t\cap\{x: M_{\mathcal{D}}v(x)>a^k\}).
\end{align*}
From Lemma~\ref{lema_estimacion_para_vt},
\begin{align*}
\s{(k,j)\in I(t,s)}{}\frac{v_k(Q^k_j)}{|Q^k_j|}u(Q^k_j)&\leq C\Phi(a)[v]_{A_1}^r\frac{u(Q_s^t)}{|Q_s^t|}\s{(k,j)\in I(t,s)}{}a^{(k-t)\alpha r}v_t(Q_s^t)a^{(t-k)r\eta}\\
&\leq C\Phi(a)[v]_{A_1}^r\frac{u(Q_s^t)}{|Q_s^t|}v_t(Q_s^t)\s{k\geq t}{}a^{(t-k)r(\eta-\alpha)}\\
&\leq C\Phi(a)[v]_{A_1}^r\frac{u(Q_s^t)}{|Q_s^t|}v_t(Q_s^t),
\end{align*}
since $\eta-\alpha>0$ and $a>2^n>1$. Thus, we have obtained that
\[\s{(k,j)\in I(t,s)}{}\frac{v_k(Q^k_j)}{|Q^k_j|}u(Q^k_j)\leq C\Phi(a)[v]_{A_1}^r\frac{u(Q_s^t)}{|Q_s^t|}v_t(Q_s^t),\]
and if we sum over all $(t,s)\in P$ it follows that
\[\s{(k,j)\in \Gamma_N}{}\frac{v_k(Q^k_j)}{|Q^k_j|}u(Q^k_j)\leq\s{(t,s)\in P}{}\s{(k,j)\in I(t,s)}{}\frac{v_k(Q^k_j)}{|Q^k_j|}u(Q^k_j)\leq C\s{(t,s)\in P}{}\frac{v_k(Q_s^t)}{|Q_s^t|}u(Q_s^t).\]
\end{proof}

\bigskip

\begin{proof}[Proof of Claim~\ref{afirmacion2}]
Assume for the moment that there exists a positive number $\gamma$ such that if $(\ell,j)\in P_\ell$ and $k_m\leq \ell<k_{m+1}$ then
\begin{equation}\label{ec1_demo_afirmacion2}
\frac{1}{|Q_j^\ell|}\int_{Q_j^\ell}u(y)\,dy>\frac{a^{(\ell-k_m)\gamma}}{2[u]_{A_1}}\frac{1}{|\tilde Q^\ell|}\int_{\tilde Q^\ell}u(y)\,dy.
\end{equation}
Then, if $y\in Q_j^\ell$,
\begin{align*}
u(y)[u]_{A_1}&\geq\frac{1}{|Q_j^\ell|}\int_{Q_j^\ell}u(z)\,dz\\
&>\frac{a^{(\ell-k_m)\gamma}}{2[u]_{A_1}}\frac{1}{|\tilde Q^\ell|}\int_{\tilde Q^\ell}u(y)\,dy
\end{align*}
and thus,
\[u(y)>\frac{a^{(\ell-k_m)\gamma}}{2[u]_{A_1}^2}\frac{u(\tilde Q^\ell)}{|\tilde Q^\ell|}=:\lambda,\]
which implies that
\[\bigcup_{j:(\ell,j)\in P_\ell}Q_j^\ell\subseteq\{x\in \tilde Q^\ell: u(x)>\lambda\}.\]
Since $u\in A_1\subseteq A_{\infty}$, there exist positive constants $C$ and $\nu$ for which $\frac{u(E)}{u(Q)}\leq C\left(\frac{|E|}{|Q|}\right)^{\nu}$ holds, for every measurable $E\subseteq Q$. So, by Chebyshev's inequality and the definition of $\lambda$ we have
\begin{align*}
\s{j:(\ell,j)\in P_\ell}{}u(Q_j^\ell)&=u\left(\bigcup_{j:(\ell,j)\in P_\ell} Q_j^\ell\right)\\
&\leq u(\{x\in \tilde Q^\ell: u(x)>\lambda\})\\
&\leq Cu(\tilde Q^\ell)\left(\frac{|\{x\in \tilde Q^\ell: u(x)>\lambda\}|}{|\tilde Q^\ell|}\right)^{\nu}\\
&\leq Cu(\tilde Q^\ell)\left(\frac{1}{\lambda|\tilde Q^\ell|}\int_{\tilde Q^\ell}u(y)\,dy\right)^{\nu}\\
&=Cu(\tilde Q^\ell)a^{(k_m-\ell)\gamma\nu},
\end{align*}
and finally
\begin{align*}
\s{\ell\in F_m}{}\s{j:(\ell,j)\in P_\ell}{}\frac{u(Q_j^\ell)}{u(\tilde Q^\ell)}&\leq C\s{\ell\in F_m}{}a^{(k_m-\ell)\gamma\nu}\\
&\leq C\s{\ell\geq k_m}{}a^{(k_m-\ell)\gamma\nu}=C,
\end{align*}
since $a>1$. This gives us the so desired proof.

\bigskip

We now prove that~\eqref{ec1_demo_afirmacion2} actually holds. Pick $(\ell,j)\in P_\ell$ with $k_m\leq \ell<k_{m+1}$. Since $\Omega_\ell\subseteq \Omega_{k_m}$, by maximality there exists a unique $s$ such that $Q_j^\ell\subseteq Q_s^{k_m}$. We want to see that $(k_m,s)\in \Gamma_N$. If $(k_m,s)\in P$ we are done since $P\subseteq \Gamma_N$. Then, let us assume that $(k_m,s)\not\in P$. From the definition of $G$ and $P_{k_m}$, $\tilde Q^{k_m}$ contains a cube $Q_p^{k_m}$ with $(k_m,p)\in P$. We shall see, as a first step, that $Q_s^{k_m}\subsetneq \tilde Q^{k_m}$. Indeed, there exists a unique $i(s)$ such that $Q_j^\ell\subseteq Q_s^{k_m}\subseteq \tilde Q_{i(s)}^{k_m}$. Also,
\[\left\{x: M_{\Phi,\mathcal{D}}g(x)>a^\ell\right\}\subseteq \left\{x: M_{\Phi,\mathcal{D}}g(x)>a^{k_m}\right\}=\bigcup_i \tilde Q_i^{k_m},\]
so there exists a unique $i_0$ such that $Q_j^\ell \subseteq \tilde Q^\ell\subseteq \tilde Q_{i_0}^{k_m}$. Besides, from the definition of $\tilde Q^k$, $x\in \tilde Q^{k_m}$ and $x\in \tilde Q_{i_0}^{k_m}$ so we must have
\[\tilde Q_{i(s)}^{k_m}=\tilde Q_{i_0}^{k_m}=\tilde Q^{k_m},\]
which directly implies $Q_s^{k_m}\subseteq \tilde Q^{k_m}$. In fact, this inclusion is proper since both $Q_s^{k_m}$ and $Q_p^{k_m}$ are contained in $\tilde Q^{k_m}$ and $s\neq p$.

Now observe that $\tilde Q^{k_m}$ is a maximal cube of the set $\{x: M_{\Phi,\mathcal{D}}g(x)>a^{k_m}\}$ and $Q_s^{k_m}$ is a maximal cube of
\[\Omega_{k_m}=\{x\in\R^n: M_{\mathcal{D}}v(x)>a^{k_m}\}\cap\{x\in\R^n: M_{\Phi,\mathcal{D}}g(x)>a^{k_m}\},\]
 and since $Q_s^{k_m}\subsetneq \tilde Q^{k_m}$ it follows that $Q_s^{k_m}$ is a maximal dyadic cube of $\{x: M_{\mathcal{D}}v(x)>a^{k_m}\}$. Thus
\begin{equation}\label{ec2_demo_afirmacion2}
\frac{1}{|Q_s^{k_m}|}\int_{Q_s^{k_m}}v(y)\,dy\leq 2^na^{k_m}\leq a^{k_m+1},
\end{equation}
so that $|Q_s^{k_m}\cap\{x:v(x)\leq a^{k_m+1}\}|>0$. Indeed, if not, denoting $E=Q_s^{k_m}\cap\{x: v(x)>a^{k_m+1}\}$ we would have
\[\frac{1}{|Q_s^{k_m}|}\int_{Q_s^{k_m}}v(y)\,dy=\frac{1}{|E|}\int_{Q_s^{k_m}}v(y)\,dy>\frac{1}{|E|}\int_E v(y)\,dy>a^{k_m+1},\]
which contradicts~\eqref{ec2_demo_afirmacion2}. Therefore, $(k_m,s)\in \Gamma_N$ and $Q_s^{k_m}$ is contained in, at least, one principal cube. Let $Q_{\sigma}^k$ the smallest principal cube that contains
$Q_s^{k_m}$. By using conditions~\eqref{desigualdad1_conjuntoGn} and~\eqref{desigualdad2_conjuntoGn} we can write
\[\frac{1}{|Q_j^\ell|}\int_{Q_j^\ell}u(y)\,dy>a^{(\ell-k)\alpha r}\frac{b_k}{b_\ell}\frac{1}{|Q_\sigma^k|}\int_{Q_\sigma^k}u(y)\,dy\geq a^{(\ell-k_m)\alpha r}\frac{b_{k_m}}{b_\ell}\frac{1}{|Q_s^{k_m}|}\int_{Q_s^{k_m}}u(y)\,dy\]

Also, from \eqref{desigualdad2_sucesionkm}
\[\frac{1}{|\tilde Q^\ell|}\int_{\tilde Q^\ell}u(y)\,dy\leq \frac{2}{|\tilde Q^{k_m}|}\int_{\tilde Q^{k_m}}u(y)\,dy\leq 2[u]_{A_1}\inf_{\tilde Q^{k_m}}u\leq2[u]_{A_1}\frac{1}{|Q_s^{k_m}|}\int_{Q_s^{k_m}}u(y)\,dy.\]
With these two inequalities we obtain the following estimate
\begin{align*}
\frac{1}{|Q_j^\ell|}\int_{Q_j^\ell}u(y)\,dy&>a^{(\ell-k)\alpha r}\frac{b_k}{b_\ell}\frac{b_{k_m}}{b_k}a^{(k-k_m)\alpha r}\frac{1}{|Q_s^{k_m}|}\int_{Q_s^{k_m}}u(y)\,dy\\
&=a^{(\ell-k_m)\alpha r}\frac{b_{k_m}}{b_\ell}\frac{1}{|Q_s^{k_m}|}\int_{Q_s^{k_m}}u(y)\,dy\\
&>\frac{1}{2[u]_{A_1}}\frac{a^{(\ell-k_m)\alpha r}}{(\Phi(a))^{\ell-k_m}}\frac{1}{|\tilde Q^\ell|}\int_{\tilde Q^\ell}u(y)\,dy,
\end{align*}
by Lemma~\ref{lema_propiedad_bk}.
We now apply Lemma~\ref{relacion_phi_con_potencia} with $0<\beta<r(\alpha-1)$. So, for $t\geq 1$ we have that $\Phi(t)\leq C_0t^{r+\beta}$. Thus,

\[\frac{1}{|Q_j^\ell|}\int_{Q_j^\ell}u(y)\,dy>\frac{1}{2[u]_{A_1}}\frac{a^{(\ell-k_m)(\alpha r-r-\beta)}}{C_0^{\ell-km}}\frac{1}{|\tilde Q^\ell|}\int_{\tilde Q^\ell}u(y)\,dy.\]
Recalling that $a>\max\{2^n,L\}$, if $\delta>0$ then $a^{\alpha r-r-\beta}/C_0>1$ by choosing $L=(\delta/\beta)^{\delta/(r(\alpha-1)-\beta)}$. Indeed, if $C_0=1$ the previous inequality trivially holds. If  not, $C_0=(\delta/\beta)^{\delta}$ and the inequality holds if and only if $a>(\delta/\beta)^{\delta/(r(\alpha-1)-\beta)}$. So, if we denote $\theta=a^{\alpha r-r-\beta}/C_0$ then we can take $\gamma=\log_a\theta$ in \eqref{ec1_demo_afirmacion2}.

If $\delta=0$, then $\Phi(t)=t^r$ and $a^{(\ell-k_m)\alpha r}/a^{(\ell-k_m)r}=a^{(\ell-k_m)(\alpha-1)r}$ and in this case we can take $\gamma=(\alpha-1)r$, and
 \eqref{ec1_demo_afirmacion2} is proved.
\end{proof}

\medskip

\section*{Acknowledgments} This paper is a constitutive part of my doctoral thesis, under the direction of Ph. D. Gladis Pradolini and Ph. D. Marilina Carena. I would like to specially thank to both of them for reading this manuscript and suggest me adequate changes in notation and redaction.


\providecommand{\bysame}{\leavevmode\hbox to3em{\hrulefill}\thinspace}
\providecommand{\MR}{\relax\ifhmode\unskip\space\fi MR }
\providecommand{\MRhref}[2]{%
  \href{http://www.ams.org/mathscinet-getitem?mr=#1}{#2}
}
\providecommand{\href}[2]{#2}

\end{document}